\documentclass[10pt,a4paper]{article}
\usepackage[english]{babel}
\usepackage{amsmath}
\usepackage{amsxtra, amsfonts, amssymb, amstext}
\usepackage{amsthm}
\usepackage{xspace}
\usepackage[noadjust]{cite}
\usepackage{url}
\usepackage{color}
\usepackage{multirow}
\usepackage{float}
\usepackage{enumerate}
\usepackage[symbol]{footmisc}
\usepackage[linkcolor=black,colorlinks=true,citecolor=black,filecolor=black]{hyperref}

\usepackage[utf8]{inputenc}
\usepackage{mathtools}
\usepackage{algorithm}
\usepackage{algpseudocode}

\DeclarePairedDelimiter\dist{\lVert}{\rVert}
\DeclarePairedDelimiter{\ceil}{\lceil}{\rceil}

\newcommand{\Vol}{\textsc{vol}}

\renewcommand{\epsilon}{\ensuremath\varepsilon}
\newcommand{\eps}{\ensuremath{\varepsilon}}
\renewcommand{\phi}{\ensuremath{\varphi}}

\newcommand{\cc}{\textsc{cc}}
\newcommand{\Ex}{\mathbb{E}}

\makeatletter
\newtheorem*{rep@theorem}{\rep@title}
\newcommand{\newreptheorem}[2]{%
\newenvironment{rep#1}[1]{%
 \def\rep@title{#2 \ref{##1}}%
 \begin{rep@theorem}}%
 {\end{rep@theorem}}}
\makeatother

\newtheorem{theorem}{Theorem}[section]

\newtheorem{corollary}[theorem]{Corollary}
\newtheorem{definition}[theorem]{Definition}
\newtheorem{lemma}[theorem]{Lemma}

\newreptheorem{theorem}{Theorem}
\newreptheorem{lemma}{Lemma}

\usepackage{enumerate}

\begin{document}

\title{Existence of Small Separators Depends on Geometry for Geometric Inhomogeneous Random Graphs}

\author{Johannes Lengler,
Lazar Todorovi\'c}
\date{
\ \\*[-2em]
ETH Z\"urich, Department of Computer Science, Switzerland\\
\bigskip
}

\maketitle

\begin{abstract}
We show that Geometric Inhomogeneous Random Graphs (GIRGs) with power law weights may either have or not have separators of linear size, depending on the underlying geometry. While it was known that for Euclidean geometry it is possible to split the giant component into two linear size components by removing at most $n^{1-\eps}$ edges, we show that this is impossible if the geometry is given by the minimum component distance.
\end{abstract}

\section{Introduction}
\emph{Geometric Inhomogeneous Random Graphs (GIRGs)} were introduced by Bringmann, Keusch and Lengler~\cite{bringmann2017sampling,BKLeuclidean} as a model for graphs which are \emph{scale-free} (i.e. the degree sequence follows a power law), but also have a community structure. In the model, each vertex draws independently a weight and a position in some geometric space, and for each pair of vertices an edge is inserted independently with a probability that depends on the weights and distance of the endpoints. The model was inspired by Chung-Lu random graphs~\cite{ChungLu}, and it contains hyperbolic random graphs (HypRGs)~\cite{KrioukovPKVB} as a special case~\cite{BKLeuclidean}. It falls into a class of Chung-Lu type random graph models~\cite{BKLgeneral}, which implies that w.h.p.\footnote{\emph{with high probability}, i.e., with probability tending to one as the number of vertices tends to $\infty$.} a random graph from the model has a giant component of linear size, all other components are of polylogarithmic size, the diameter of each component is polylogarithmic (small-world), and the average distance in the giant component is $O(\log \log n)$ (ultra-small world). On the other hand, for the case of euclidean GIRGs (EuGIRGs, the underlying geometric space has euclidean geometry) it was shown~\cite{bringmann2017sampling,BKLeuclidean} that w.h.p. the graphs have clustering coefficient\footnote{Essentially, the clustering coefficient is the probability that two random neighbours of a random vertex are adjacent, see Definition~\ref{def:cc} on page~\pageref{def:cc} for a precise definition.} $\Omega(1)$, and that they have sublinear separators.\footnote{I.e., the giant component may be split in two linear size parts by removing a sublinear number of edges.} 

Other geometries, in particular the non-metric \emph{Minimum Component Distance} (MCD,~\cite[Example 7.2]{BKLgeneral}), have been proposed, but not intensively studied. It is clear from~\cite{BKLgeneral,bringmann2017sampling,BKLeuclidean} that the results on component structure, diameter, average distance, and clustering coefficient carry over.\footnote{For the clustering coefficient this was not explicit in~\cite{bringmann2017sampling,BKLeuclidean}. For convenience, we give a proof in the appendix.} However, it was unclear whether small separators exist in these graphs. In this paper we give a negative answer: w.h.p. the giant component cannot be split in two linear size parts by removing a sublinear number of edges. Thus in the GIRG model, it depends on the underlying geometry whether small separators exists.

For Erd\H{o}s-R\'enyi random graphs with edge probability $p=\tfrac{1+\Omega(1)}{n}$, there is a classical technique by {\L}uczak and McDiarmid~\cite{LMD} to prove that the giant component is stable, i..e that there are no sublinear separators. The main idea is to draw the edges in two batches with probabilities $p_1$ and $p_2$, respectively, where $p_2 = \tfrac{\eps}{n}$ for some small $\eps >0$, and $p_1 \approx p-p_2$. After the first batch there is already a giant component, and {\L}uczak and McDiarmid prove that for any connected graph there are relatively few sparse cuts, i.e., cuts with a small number of cross-edges, see Lemma~\ref{LMD} below for the exact statement. Since the second batch is drawn independently of the first one, all these sparse cuts will be destroyed by the second batch. Moreover, since $p_2$ is small, the giant component is not changed significantly by the second batch. Bollob\'as, Janson, and Riordan~\cite{bollobas2007phase} have shown that the proof can be generalised to a much more general class of random graph models.\footnote{This includes Chung-Lu random graphs, but not GIRGs; the latter is most evident by observing that GIRGs have different properties than the graphs in~\cite{bollobas2007phase}, for example EuGIRGs have sublinear separators.} A crucial ingredient for this approach is that the edges can be inserted in \emph{independent} batches. (Or alternatively, graph exposition can be linked to a branching process.) This is not the case for GIRGs, since the edges depend on the position of the vertex, and are thus highly dependent of each other. The \emph{main technical contribution} of this paper is how these dependencies can be overcome. An essential ingredient is a version of Azuma's inequality with two-sided error that has recently been developed independently (in slightly different forms) by Bringmann, Lengler, and Keusch~\cite[Theorem 3.3]{BKLgeneral}, and by Combes~\cite{combes15}. We refer the reader to Section~\ref{secComp} for a more detailed overview over the proof and its main obstacles. 

\paragraph{Background on the euclidean GIRG model.} The EuGIRG model, and in particular the HypRG model\footnote{In hyperbolic random graphs the positions are drawn uniformly at random in a hyperbolic disc. This is a special case of one-dimensional GIRGs, where the angular coordinate of the hyperbolic plane corresponds to a one-dimensional geometric space, and the radial coordinate corresponds to the weight, cf.~\cite[Section 7]{bringmann2017sampling}.}, have been intensively studied in the last years in mathematics~\cite{bode2015largest,candellero2016clustering,candellero2016bootstrap,BKLgeneral,kiwi2017spectral}, physics~\cite{BogunaPK10,KrioukovPKVB,BogunaK09,BogunaKC09}, and computer science~\cite{KrioukovPBV09,PapadopoulosKBV10,gugelmann2012random,kiwi2014bound,friedrich2015cliques,papadopoulos2015network,koch2016bootstrap,blasius2016efficient,blasius2016hyperbolic,bringmann2017sampling,lengler2017greedy}. A closely related model of infinite graphs is \emph{scale-free percolation}~\cite{deijfen2013scale,deprez2015inhomogeneous}, see~\cite{Hofstad2017explosion} for a comparison. EuGIRGs have been used as a model for social networks~\cite{lengler2017greedy} and for the internet graph~\cite{BogunaPK10,PapadopoulosKBV10}, especially for evaluating routing protocols, both experimentally~\cite{BogunaPK10,KrioukovPKVB,BogunaK09,KrioukovPBV09,PapadopoulosKBV10,papadopoulos2015network,blasius2016efficient} and theoretically~\cite{lengler2017greedy}. Mathematically, EuGIRGs have been studied with respect to their separators and entropy~\cite{bringmann2017sampling}, their spectral properties~\cite{kiwi2017spectral}, and under bootstrap percolation~\cite{candellero2016bootstrap,koch2016bootstrap}. Computationally, sampling algorithms~\cite{kiwi2014bound,bringmann2017sampling} and compression algorithms~\cite{bringmann2017sampling} have been studied.

The analysis of GIRGs is often considerably harder than the analysis of non-geometric random graph models like the Erd\H{o}s-R\'enyi model or the Chung-Lu model. Most notably, an important and classic tool for the latter are branching processes. For example, exploring an Erd\H{o}s-R\'enyi graph with a breadth-first search is intimately linked to a branching process, since for a long time any new edge will lead to an unexplored vertex w.h.p.. This connection fails for GIRGs\footnote{although recently a related approach has been proposed in~\cite{Hofstad2017explosion} for scale-free percolation.} because they are strongly clustered. Hence the probability that a new edge leads to an already explored vertex is rather large and, crucially, depends strongly on the positions of the uncovered vertices. This broken connection leads to many graph properties which are not observed in non-geometric graphs. For example, a branching process is very unlikely to reach $\omega(\log n)$ leaves and still die out (except in threshold cases), which is why there are no components in the range $\omega(\log n) \cap o(n)$ in non-threshold Erd\H{o}s-R\'enyi graphs or Chung-Lu graphs. On the other hand, GIRGs are easily seen to have components of size $(\log n)^{1+\Theta(1)}$ w.h.p.\footnote{In fact, quite a lot of them. Every box of volume $\tfrac1n (\log n)^{1+\eps}$ has probability $n^{o(1)}$ to contain such a component if $\eps>0$ is a sufficiently small constant, so the expected number of such components is $n^{1-o(1)}$.}, which shows that they behave very different from branching processes. 

On the other hand, GIRGs have still a certain degree of independence among the edges: conditioned on the position of the vertices, the edges are generated by independent coin flips. This still allows to use concentration inequalities in some settings. For example, for any fixed vertex $v$, the degree of $v$ is the sum of $n-1$ independent random variables (and thus, approximately Poisson distributed). However, for a pair of vertices $u,v$ the situation is already much more involved: if $u$ and $v$ are adjacent, then their degrees in the remaining graph are still Poisson distributed (in the limit), but they are not distributed like two \emph{independent} Poisson distributions.


\paragraph{Background on the Minimum Component Distance (MCD).} In the GIRG model, the geometric position of a vertex is supposed to represent properties and preferences of that vertex, like place of residence, profession, interests, and hobbies for a social network. It has been argued that the euclidean distance does not well reflect the underlying geometry of social networks, since two vertices are only close in the euclidean geometry if they are similar in \emph{all} their components, while nodes in social networks should be likely to connect if they agree in \emph{some} aspects. Thus other distance functions have been proposed, in particular the \emph{Minimum Component Distance} (MCD~\cite[Example 7.2]{BKLgeneral}), where the distance between two $d$-dimensional vectors is the minimum of the distances in each components. Note that this distance function is not a metric, which is also considered realistic for social networks. The resulting model is still of Chung-Lu type~\cite[Theorem 7.3]{BKLgeneral}, and the results in~\cite{BKLgeneral} still apply (component structure, diameter, average distance). Moreover, it is easy to see that the MCD still has a clustering coefficient of $\Omega(1)$ (see appendix), which is considered important for a model of social and technological networks. Non-geometric random graph models like (sparse) Erd\"os-R\'enyi graphs or Chung-Lu graphs tend to have very small clustering coefficients~\cite{van2016random}. 

\paragraph{Structure of the Paper}
In Section~\ref{secModel} we will give a formal definition of the GIRG model. In Section~\ref{secResult} we formally define $(\delta,\eta)$-cuts and gives the formal statement of our main theorem that MCD-GIRGs have no small separators. In Section~\ref{secComp} we give a proof outline, and explain why the classical approach of {\L}uczak and McDiarmid is not directly applicable. Section~\ref{secTools} collects some tools and describes the uncovering procedure that we use in our proof, while~\ref{proofSec} contains the proof of the main theorem. Finally, in the appendix we show that w.h.p.\! MCD-GIRGs have clustering coefficient $\Omega(1)$ as the MCD satisfies a stochastic triangle inequality.

\section{Model}
\label{secModel}
We now come to the definition of the GIRG model. The GIRG model was first introduced in~\cite{bringmann2017sampling}, but only for euclidean distances\footnote{or for any other metric that comes from a norm -- they all give equivalent models.}, and we refer to this case as \emph{euclidean GIRG} or \emph{EuGIRG}. We follow the definition in \cite[Section 7]{BKLgeneral}, which allows a much wider class of distance functions.

The GIRG model is defined on the vertex set $V = \{1, \ldots, n\}$,\footnote{where we are interested in the case $n\to\infty$. All Landau notation $O(.), \Omega(.)$ etc. is w.r.t.\! this limit, and likewise we mean by ``a constant'' a quantity that does not depend on $n$.} and comes with the following set of parameters.
\begin{itemize}
\item A geometric ground space, which we assume to be the $d$-dimensional torus $\mathbb{T}^d = \mathbb{R}^d / \mathbb{Z}^d$ for some constant $d \geq 1$.
\item A (non-necessarily metric) distance function on $\mathbb{T}^d$, cf.~below.
\item For every $n$ a sequence $\mathbf{w} = (w_1,\ldots,w_n)$ of $n$ positive weights\footnote{we suppress from the notation that $\mathbf{w}$ depends on $n$}, such that a weak power law condition for some constant power law exponent $\beta \in (2,3)$ is satisfied, see~\eqref{PL1} and~\eqref{PL2} below.
\item For each $n$ a function $p = p(x_u,x_v,w_u,w_v) \in [0,1]$ that assigns to each quadruple of two positions and two weights a probability and that satisfies~\eqref{EP} below for some constant parameter $\alpha >1$.
\end{itemize}
Before we elaborate on the precise requirements of the parameters, let us define the model. A random graph in the GIRG model on the vertex set $V=\{1,\ldots,n\}$ is obtained by the following procedure. Each vertex $v$ is equipped with \emph{weight $w_v$}, and every vertex draws independently a uniformly randomly \emph{position $x_v \in \mathbb T^d$}. Afterwards, for each pair $1 \leq u <v \leq n$ of different vertices we independently flip a coin, and insert the edge $\{u,v\}$ with probability $p_{uv} := p(x_u,x_v,w_u,w_v)$.


\subsection{Requirements of the GIRG model.}
\paragraph*{Distance Function.}
For simplicity we require the distance function to be translation invariant and symmetric. I.e., we can describe it by a measurable function $\dist{.} : \mathbb{T}^d \rightarrow \mathbb{R}_{\geq 0}$, where $\dist{x} = \dist{-x}$ for all $x\in\mathbb{T}^d$, and 
$\dist{0} = 0$. Then we define the distance of $x,y \in \mathbb{T}^d$ as $\dist{x-y}$, where $x-y \in \mathbb{T}^d$ is taken in the additive group $(\mathbb{T}^d,+)$. We additionally require that the mapping $V: \mathbb{R}_0^+ \rightarrow [0,1]; V(r) := \Vol(\{x \in 
\mathbb{T}^d | \, \dist{x} \leq r \})$ is surjective.

Note crucially that $\dist{.}$ does not need to be a norm.

\paragraph*{Edge probabilities.}
Let $\alpha > 1$ be a constant.\footnote{This constant essentially enforces convergences. Variations with $\alpha \leq 1$ are possible, but make the requirements on the model more technical, see~\cite{BKLgeneral}. Similarly, in~\cite{BKLeuclidean} also the threshold case $\alpha = \infty$ was considered. We omit both variants for the sake of simplicity.} We require that the edge probability function $p(x_u, x_v,w_u,w_v)$ satisfies for all $x_u \neq x_v \in \mathbb{T}^d$ and all $w_u,w_v > 0$:
\begin{equation}
 p(x_u, x_v,w_u,w_v) = \Theta \left( \min \left\{ 1, 
\left( \frac{w_uw_v}{n\cdot V(\dist{x_u - x_v})}\right)^\alpha \right\}\right), \label{EP} \tag{EP}
\end{equation}
where as before we denote $V(r) := \Vol(\{x \in \mathbb{T}^d | \, \dist{x} \leq r \})$. Since only $x_u$ and $x_v$ are random, we also write $p_{uv} := p_{uv}(x_u,x_v) := p(x_u, x_v,w_u,w_v)$.

The bounds represented by $\Theta$ are assumed to be global, i.e. there are
constants $c_L, c_U > 0$ such that for sufficiently large $n$ the inequalities
\begin{align}
p_{uv}(x_u, x_v) & \geq c_L \min \left\{1,\left( \frac{w_uw_v}{n\cdot V(\dist{x_u - x_v})}\right)^\alpha \right\} \text{ and} \label{EPL} \tag{EPL}\\
p_{uv}(x_u, x_v) &\leq c_U  \min \left\{1, \left( \frac{w_uw_v}{n\cdot V(\dist{x_u - x_v})}\right)^\alpha \right\}             \label{EPU} \tag{EPU}
\end{align}
hold for all $x_u \neq x_v \in \mathbb{T}^d$ and all $w_u,w_v > 0$.

We remark that~\eqref{EP} implies that a vertex $v$ with weight $w_v$ has expected degree $\Ex[\deg(v)] = \Theta(w_v)$, and two vertices $u,v$ with weights $w_u,w_v$, respectively, have probability $\Pr[\{u,v\} \in E] = \Theta(\min\{1,w_uw_v/n\})$ to be adjacent~\cite{BKLgeneral}.

\paragraph*{Power Law Weights.\footnote{We remark that all the proofs in the paper go through under much weaker assumptions. More precisely, the power law condition~\eqref{PL2} could be replaced by the weaker requirement that $\tfrac1n \sum{w_i} = O(1)$ and $\tfrac1n \sum{w_i^2} \to \infty$. However, since we want to cite results from~\cite{BKLgeneral} (in particular, the existence of a giant component), we stick with the conditions from~\cite{BKLgeneral}. Note that some properties that make the model popular would not hold without power law weights, for example the graphs would not be (ultra-)small worlds in general.}}
Let $2 < \beta < 3$ be a constant. We demand that
\begin{equation}
w_\text{min} := \min \{w_v \mid v \in V \} = \Omega(1), \label{PL1} \tag{PL1} 
\end{equation}
and that there exists $\bar{w} = \bar{w}(n) \geq n^{\omega(1/\log \log n)}$ 
such that for all constants $\eta > 0$ there are $c_1, c_2 > 0$ such that for all sufficiently large $n$,
\begin{equation}
c_1\frac{n}{ w^{\beta -1 + \eta} } \leq \#\{v \in V |\, w_v \geq w \} \leq c_2 
\frac{n}{ w^{\beta - 1 - \eta} }, \label{PL2} \tag{PL2}
\end{equation}
where the second inequality holds for all $w \geq w_\text{min}$, and the first one 
only holds for all $w_\text{min} \leq w \leq \bar{w}$. These conditions directly
imply that the total weight $W := \sum_{v \in V} w_v$ is in $\Theta(n)$~\cite[Lemma 4.2]{BKLgeneral}. 



\subsection{Minimum Component Distance.}
Since we are dealing with a torus, we define the absolute difference for 
$a, b \in [0,1]$ as
\begin{equation}
|a-b|_T := \min\{|a-b|, 1 - |a-b|\}.
\end{equation}
For $x,y \in \mathbb{T}^d$, the minimum component distance MCD can then be expressed as
\begin{align}
\dist{x-y}_\text{min} &:= \min_{1\leq i \leq d} |x_i -y_i|_T.
\end{align}
Note that the MCD coincides with the euclidean distance on $\mathbb{T}^1$ 
and that MCD is not a norm for $d > 1$, since the triangle inequality does not hold. While the volume function $V(r)$ for the euclidean norm
grows as $V_\text{eu}(r) = \Theta(r^d)$ for $0\leq r \leq 1/2$, for MCD it is given by $1 - (1-2r)^d$.
As this is somewhat unwieldy, we observe that $2r \leq 1 - (1-2r)^d \leq 2dr$.
Therefore, if we somewhat abusively redefine $V_\text{min}(r) := r$, we are off by at most a
constant factor; this does not change the model, because such constant
factors are swallowed for by the $\Theta$-notation in~\ref{EP}.


\subsection{Notation}
We denote the $i$-th component of a vector $e$ by $e_i$. In particular, for the position $x_v \in \mathbb{T}^d$ of vertex $v$ we denote by $x_{vi}$ the $i$-th component, for $1 \leq i \leq d$. For $r \geq 0$ and $x \in \mathbb{T}^d$, let $B_r(x) := \{y \in 
\mathbb{T}^d | \, \dist{x-y} \leq r \}$ denote the $r$-ball around $x$ with 
respect to $\dist{.}$, and we let $V(r) := r = \Theta(\Vol(B_r(x)))$ whenever we are concerned with MCD. We denote by $w_\text{min}$ the minimum weight, and by $W := \sum_{v\in V} w_v$ the 
total weight.
For convenience, we remind the reader that $\alpha >1$ is the convergence-enforcing constant of the GIRG model, and that $\beta \in (2,3)$ is the power law exponent. The position and weight of a vertex $v$ are denoted by $x_v$ and $w_v$, respectively. The connection probability of two vertices $u$ and $v$ is $p_{uv} := p(x_u,x_v,w_u,w_v)$.

We say that a family of events $\mathcal{E}_n$ holds \emph{with high probability} (w.h.p.) if $\Pr[\mathcal E_n] \to 1$ as $n\to\infty$. If $V_1,V_2$ are two sets of vertices in a graph $G$, then we denote by $E(V_1,V_2)$ the set of edges from $V_1$ to $V_2$.

By default, in the proofs we use $n$ for the number of vertices of the GIRG that we study.

\section{Our Result}\label{secResult}
The main result of this paper is that the giant component in the MCD model is stable under the removal of a sublinear number of edges. 
It is known that w.h.p.\! euclidean GIRGs have
small (edge) separators of size $n^{1-\eps}$~\cite{BKLeuclidean}. In contrast, we will prove that for MCD-GIRGs, w.h.p.\! the removal of any sublinear number of edges only decreases the size of the largest component by a sublinear number of vertices. In particular, this shows that the existence of small separators in the GIRG model depends on the properties of the underlying geometry. 

In order to formulate the precise statement, we introduce some
terminology.
\begin{definition}
For a graph $G = (V,E)$ and $\delta, \eta > 0$, a \emph{$(\delta, \eta)$-cut} is
a partition of $V$ into two sets of size at least $\delta|V|$ such that there
are at most $\eta|V|$ \emph{cross-edges} (i.e. edges that have one endpoint in
each of the sets).
\end{definition}
Somewhat loosely speaking,
we say that a family of graphs has \emph{small separators} if for a graph with $n$ vertices
it is possible to split its giant component into parts of linear size 
by removing only $o(n)$ edges. Note that
we always refer to splitting the giant component and not the whole graph;
this is necessary because a GIRG of size $n$ has $\Omega(n)$
isolated vertices w.h.p., yielding a trivial cut between isolated and non-isolated
vertices. According to Theorem \ref{compSize} below, w.h.p. there is only one component of size $\Omega(n)$, which justifies the term ``giant component''.

More formally, we will consider $(\delta, \eta)$-cuts of the largest component of $G=(V,E)$.
We will use the convention that the two sets into which that component is partitioned
still need to have size at least $\delta |V|$, and the cut must have at most $\eta |V|$
cross-edges (i.e. we do not use the size of the component instead of $|V|$).

Our main theorem states that w.h.p.\! MCD-GIRGs for dimension $d\geq 2$ have no small separators. 
\begin{theorem}[Main Theorem]
\label{mainThm}
Let $G$ be a graph on $n$ vertices drawn according to the GIRG model with minimum component distance and with dimension $d\geq 2$. Then, for every
$\delta > 0$ there is an $\eta > 0$ such that w.h.p.\! the largest
component of $G$ has no $(\delta, \eta)$-cut.
\end{theorem}
Note that for $d=1$, the minimum component distance agrees with the euclidean distance. Hence, for $d=1$ there exist sublinear separators.\footnote{For euclidean GIRGs, the separator is constructed by taking two parallel, axis-parallel half-planes, and removing all edges across these two half-planes. By~\cite[Theorem 2.5]{BKLeuclidean}, w.h.p. the number of edges across each axis-parallel half-plane is at most $n^{1-\eps +o(1)}$ for $\eps := \min\{\alpha -1, \beta-2, 1/d\}$.}

\section{Proof Outline}
\label{secComp}

In this section we will review the classic approach by {\L}uczak and McDiarmid~\cite{LMD} for Erd\H{o}s-R\'enyi random graphs $G(n,p)$ with edge probability $p=c/n$ where $c>1$ is a constant (for $c<1$ there is no giant component). Then we will explain why the proof cannot be straightforwardly transferred to our setting, and how we can overcome the obstacles. A key ingredient for both proofs is the following lemma, which states that in every connected graph there are only relatively few sparse cuts.
\begin{lemma}[{Cut bound~\cite[Lemma 7]{LMD}}]
\label{LMD}
For every $\varepsilon > 0$ there exist $\eta > 0$ and $n_0$ such 
that the following holds. For all $n \geq n_0$, and for all connected graphs 
$G$ with $n$ vertices, there are at most $(1+\varepsilon)^n$ bipartitions of 
$G$ with at most $\eta n$ cross-edges.
\end{lemma}

\paragraph{Erd\H{o}s-R\'enyi Graphs.}
With Lemma~\ref{LMD}, the proof for the Erd\H{o}s-R\'enyi case becomes rather straightforward. Our goal is to show that for every $\delta > 0$, there is $\eta > 0$ such that w.h.p.\! the giant of $G$ has no $(2\delta, \eta)$-cut (the factor 2 is chosen for convenience of exposition). We fix $\gamma = \gamma(\eta,c)$ sufficiently small, and uncover the edges in two batches, where in the first batch we insert each edge with probability $p_1 := (c-\gamma)/n$, thus obtained $G'$. Afterwards, in the second batch we insert each edge with probability $p_2 \approx \gamma/n$, where $(1-p_1)(1-p_2) = (1-p)$. Note that the resulting graph is distributed as $G(n,p)$.

By Lemma~\ref{LMD}, the number of $(\delta, \eta)$-cuts of the giant component of $G'$ is at most $(1+\varepsilon)^n$. By a standard Chernoff bound, each such cut will receive at least $\eta n$ cross-edges from the second batch with probability $1-e^{-\Omega(n)}$, if $\eta >0$ is sufficiently small. Thus we can afford a union bound over all $(1+\eps)^n$ such cuts in $G'$ if $\eps>0$ is sufficiently small, and w.h.p. all of them will be destroyed in $G$ (i.e., all of them will have at least $\eta n$ cross-edges in $G$). Finally, let $K_\text{max}$ and $K_\text{max}'$ denote the vertex set of the giant component of $G$ and $G'$, respectively. If $\gamma$ is small enough then w.h.p.\! $|K_\text{max} \setminus K_\text{max}'| < \delta n$. In this case every $(2\delta,\eta)$-cut of $K_\text{max}$ in $G$ would give rise to a $(\delta, \eta)$-cut of $K_\text{max}'$ in $G$ by deleting the vertices in $K_\text{max} \setminus K_\text{max}'$. Since we have shown that w.h.p.\! the latter does not exist, there is no $(2\delta,\eta)$-cut of $K_\text{max}$ in $G$.

\paragraph{MCD-GIRG graphs}
While we will still use Lemma~\ref{LMD} for the MCD-GIRG case, the remainder of the proof does not carry over. Firstly, it is not possible to generate the edge set by two independent batches, a large one and a small one. In fact, it \emph{is} possible to generate the graph in two ``almost'' independent batches, by first uncovering all edges that ``come from'' any of the first $d-1$ components of the positions, giving a graph $G_1$, and afterwards uncovering all edges that ``come from'' the last component of the position. We will make use of this two-stage process in the proof, see~\eqref{LB1} and~\eqref{LB2} in Section~\ref{secLB}. However, unlike in the Erd\H{o}s-R\'enyi case, the second batch of edges is not small: it comprises for roughly a $1/d$ fraction of all edges, and the giant component will grow substantially in this last step. In particular, we cannot avoid that many new bipartitions of the giant component are created by the second batch. 

The basic idea is, instead of uncovering the \emph{edges} in batches, we rather want to uncover the \emph{vertices} in batches. However, we only want to use vertices for the second batch that are contained in the giant component, as otherwise the incident edges would not count towards the cut across the giant component. These objectives seem contradictive, but fortunately it is possible to reconcile them by first uncovering partial geometric information. More precisely, as indicated above we first uncover the first $d-1$ components of each position, and draw edges conservatively according to these coordinates, i.e., we never overestimate the probability of an edge to be in $G$. This gives us a graph $G_1$ which has a giant component w.h.p.. Within the vertices of bounded weight in this giant component we draw randomly a small (but linear size) subset $F$. The yet uncovered incident edges to $F$ will serve as the second batch.

Next we uncover the full geometric information of all vertices in $V\setminus F$, and insert all edges between these vertices. In the proof, this
intermediate graph is denoted by $G_3$ (as $G_2$ is used for an earlier phase
which we skip over in this outline). Finally, we draw the $d$-th coordinate for
vertices in $F$, which allows us to add all edges incident to 
$F$, completing the graph. We show that these additional edges incident to $F$, which result
from vertices being close along the $d$-th coordinate, destroy all $(\delta, \eta)$-cuts
in the giant of $G_3$ (Lemma \ref{lCut}). This is less trivial than in the Erd\H{o}s-R\'enyi case, since the edge probabilities depend on the
(already uncovered) $d$-th coordinate of the vertices in $V\setminus F$. Thus the edge probability are strongly correlated. However, in Lemma~\ref{l1} we show that
the vertices in $V \setminus F$ are spread uniformly enough over the $d$-th dimension to alleviate this issue. 


Finally, we need to show that 
the edges incident to $F$ added in the last step do not increase the size
of the giant by too much (Lemma \ref{l4}). Again, this part is significantly more involved than in the
Erd\H{o}s-Rényi case, since not even the exact size of the giant component of a GIRG is known. 
We proceed in two steps. Firstly we show that the number of edges going of from $F$ is concentrated (Lemma \ref{lEdgeNum}); this already requires an Azuma-Hoeffding type bound with two-sided error that has only recently been developed (Theorem \ref{thm18}, originally from \cite{BKLgeneral}). 
Secondly, we show that the number of vertices in large non-giant components in $G_3$ (Lemma~\ref{l2}) can be coupled to the result of a random walk if in each step the $d$-th coordinates of the uncovered vertices are sufficiently uniformly spread. In this way we derive an upper bound on the number of vertices in such components. Together, this implies that on average each edge incident to $F$
adds only a constant number of vertices to the giant. This allows us to show that w.h.p.\! the edges uncovered in the last step increase the giant component only by $\delta n$ vertices. The theorem then follows analogously to the Erd\H{o}s-R\'enyi case.

\section{Preliminaries and Tools}
\label{secTools}

\subsection{Tools}
We start with some basic tools. Concerning GIRGs, we mainly need the existence of a giant component:
\begin{theorem}[{Component Structure of GIRGs,~\cite[Theorems 5.9 + 7.3]{BKLgeneral}}]
\label{compSize}
\hfill 
\begin{enumerate}[(i)]
\item There is a constant $s_\text{max}$ such that with high probability, there 
is a connected component containing at least $s_\text{max} n = \Omega(n)$ 
vertices.
\item With high probability, all other components have at most $\log^{O(1)} n$ vertices 
(polylogarithmic size).
\end{enumerate}
\end{theorem}
Note that $s_\text{max}$ depends on the parameters of the model, including the
factors hidden by $\Theta$ in the edge probability equation \eqref{EP}.

Moreover we need that the degree distribution follows a power law. (This is not surprising but recall that the definition only requires the \emph{weights} to follow a power law.) 
\begin{lemma}[{\cite[Theorem 6.3 and Lemma 4.5]{BKLgeneral}}]\label{thm:degrees}
For all $\eta > 0$ there exist constants $c_3, c_4 > 0$ such that w.h.p.
\begin{equation}\label{eq:powerlawdegree}
c_3\frac{n}{ c^{\beta -1 + \eta} } \leq \#\{v \in V |\, \deg(v) \geq c \}
\leq c_4  \frac{n}{ c^{\beta - 1 - \eta} }.
\end{equation}
where the second inequality holds for all $c \geq 1$ and the first one hold for
all $1 \leq c \leq \bar{w}$, where $\bar{w}$ is the same as in \eqref{PL2}. Moreover, there is a constant $C>0$ such that $\Ex[\deg(v)] \leq Cw_v$, and such that with probability $1-n^{\omega(1)}$
\begin{equation}
\deg(v)\leq C\cdot (w_v + \log^2 n)
\end{equation}
holds for all $v \in V$.
\end{lemma}
Note that the exponents in \eqref{eq:powerlawdegree} are the same as for the weights (see \eqref{PL2}). 

Next we give some classical concentration inequality. In the proof we will use the following Chernoff bounds:
\begin{lemma}[{``Weak'' Chernoff bounds~\cite[Theorem 1.1]{dubhashi}}]
\label{wchb}
Let $X := \sum_{i=1}^n X_i$ where for all $1 \leq i \leq n$, the random 
variables $X_i$ are independently distributed in $[0,1]$. Then
\begin{enumerate}[(i)]
\item $\mathbb{P}[X > (1+\varepsilon)\mathbb{E}[X]] \leq \exp 
\left(-\frac{\varepsilon^2}{3} \mathbb{E}[X] \right)$, for all $0<\varepsilon<1$,
\item $\mathbb{P}[X < (1-\varepsilon)\mathbb{E}[X]] \leq \exp
\left(-\frac{\varepsilon^2}{2} \mathbb{E}[X] \right)$, for all $0<\varepsilon<1$,
and
\item $\mathbb{P}[X > t] \leq 2^{-t}$ for all $t > 2e\mathbb{E}[X]$.
\end{enumerate}
In particular, this holds when the $X_i$ are i.i.d. indicator random variables.
\end{lemma}

In one case, this simple version will not do, and we need to use the stronger form:
\begin{lemma}[{``Strong'' Chernoff bound~\cite[Theorem 2.15]{ChungLu}}]
\label{schb}
Let $X:= \sum_{i=1}^n X_i$ where for all $1 \leq i \leq n$, the random 
variables $X_i$ are independently distributed in~$\{0,1\}$ (i.e. they 
are indicator variables).
Then, for any $\epsilon > 0$ we have
\begin{equation}
\mathbb{P}[X \geq (1+\epsilon)\mathbb{E}[X] ] 
\leq \left( \frac{e^\epsilon}{(1+\epsilon)^{1+\epsilon}} \right)^{\mathbb{E}[X]}
\leq \left( \frac{e}{1+\epsilon} \right)^{(1+\epsilon)\mathbb{E}[X]}.
\end{equation}
\end{lemma}
The second inequality is a trivial simplification that will come in handy when 
we apply this lemma.

Finally, we will need a variant of the Azuma-Hoeffding bound with two-sided error event. The crucial difference to other versions of the Azuma-Hoeffding bound is that for \emph{both} the event $\omega$ and $\omega'$ we only need to consider ``good'' events.
\begin{theorem}[{Azuma-Hoeffding Bound with Two-Sided Error Events, \cite[Theorem 3.3]{BKLgeneral}}]
\label{thm18}
Let $Z_1, \ldots, Z_m$ be independent random variables over $\Omega_1, \ldots, 
\Omega_m$. Let $Z = (Z_1, \ldots, Z_m)$, $\Omega = \prod_{k=1}^m \Omega_k$ and 
let $g: \Omega \rightarrow \mathbb{R}$ be measurable with $0 \leq g(\omega) 
\leq M$ for all $\omega \in \Omega$. Let $\mathcal{B} \subseteq \Omega$ (the 
subset of ``bad'' events) such that for some $c > 0$ and for all $\omega, 
\omega^\prime \in \Omega\setminus \mathcal{B}$ that differ in at most two components 
we have
\begin{equation}
|g(\omega) - g(\omega^\prime)| \leq c.
\end{equation}
Then for all $t > 0$
\begin{equation}
\mathbb{P}\left[\lvert g(Z) - \mathbb{E}[g(Z)]\rvert \geq t \right]
\leq 2e^{-\frac{t^2}{32mc^2}} + \left(2\tfrac{mM}{c} +1\right)\mathbb{P}[\mathcal{B}]. 
\end{equation}
\end{theorem}

Finally, LeCam's inequality will imply that the degree distribution is approximatively Poisson distributed.
\begin{theorem}[LeCam~\cite{le1960approximation}]\label{thm:lecam}
  Suppose $X_1,\ldots,X_n$ are independent Bernoulli random variables s.t.\
  $\Pr[X_i=1]=p_i$ for $i\in [n]$, $\lambda_n = \sum_{i\in [n]}p_i$ and
  $S_n=\sum_{i\in [n]} X_i$. Then
  \[
    \sum_{k=0}^\infty\left|\Pr[S_n=k] - \frac{\lambda_n^k e^{-\lambda_n}}{k!}\right| < 2\sum_{i=1}^n p_i^2.
  \]
  In particular, if $\lambda_n = \Theta(1)$ and $\max_{i \in [n]}p_i = o(1)$, then $\Pr[S_n = k] = \Theta(1)$ for $k= O(1)$.
\end{theorem}

For convenience, we also restate here Lemma~\ref{LMD}, which bounds the number of small cuts in connected graphs. 
\begin{replemma}{LMD}[{Cut bound~\cite[Lemma 7]{LMD}}]
For any $\varepsilon > 0$ there exist $\eta_0(\varepsilon) > 0$ and $n_0$ such 
that the following holds. For all $n \geq n_0$, and for all connected graphs 
$G$ with $n$ vertices, there are at most $(1+\varepsilon)^n$ bipartitions of 
$G$ with at most $\eta_0 n$ cross-edges.
\end{replemma}

\subsection{Component-Wise Sampling}
\label{secLB}
Naively, we would generate a graph from the GIRG model by drawing all positions (which determines all $p_{uv}$ for $u \neq v \in V$), and then inserting the edge $\{u,v\}$ with probability $p_{uv}$. The latter can be done by drawing independent random variables $Y_{uv}$ uniformly at random from $[0,1]$, and inserting the edge $\{u,v\}$ if and only if $Y_{uv} < p_{uv}$. Since the random $Y_{uv}$, $x_u$, $x_v$ are independent, we may draw them in any order. In particular, we may draw $Y_{uv}$ \emph{before} drawing $x_u$ and $x_v$. This point of view will be helpful in several places. In particular, it will allow us to insert some edges even if only the first $d-1$ component of $x_u$ and $x_v$ have been drawn. More precisely, we proceed as follows.


Inequality \eqref{EPL} is a lower bound for $p_{uv}$; this means that
\begin{equation}
Y_{uv} < c_L \min \left\{1,\left( \frac{w_uw_v}{n \cdot V(\dist{x_u - x_v})}\right)^\alpha \right\}
\end{equation}
implies $Y_{uv} < p_{uv}$. Therefore, it is a sufficient condition for
edge insertion, i.e. the edge between $u$ and $v$ is inserted with probability
one if it is fulfilled. We are particularly interested in applying this to the
MCD model, where we have defined $V_\text{min}(r) = r$.
For the MCD model, the sufficient condition is thus given as
\begin{equation}
\label{lowerB}
\begin{split}
Y_{uv} & < c_L \min \left\{1, \left( \frac{w_uw_v}{n \cdot \dist{x_u - x_v}_\text{min}}\right)^\alpha \right\} =: p_{uv}^L(\dist{x_u - x_v}_\text{min}).
\end{split}
\end{equation}
Note that if 
$\dist{x_u-x_v}_\text{min} \leq \frac{w_uw_v}{W}$, the $\min \{\ldots  \}$ term
becomes 1. In that case, having $Y_{uv} < c_L$ guarantees that the edge between
$u$ and $u$ is inserted, i.e.
\begin{equation}
\mathbb{P}\left[ u\sim v \mid \dist{x_u - x_v}_\text{min} \leq \frac{w_uw_v}{n}\right]
\geq \min\{c_L, 1\}.
\end{equation}
Note that $w_uw_v \geq w_\text{min}^2 = \Omega(1)$ due to \eqref{PL1}. This
enables us to give a lower bound on connection probability that depends on
$\dist{x_u - x_v}_\text{min}$ but not on the weights:
\begin{equation}
\label{closeB}
\mathbb{P}\left[ u\sim v \mid \dist{x_u - x_v}_\text{min} \leq \frac{w_\text{min}^2}{n}\right]
\geq \min\{c_L, 1\}.
\end{equation}
This statement (with a slight modification) will be quite handy for proving a
key lemma, namely Lemma \ref{l1}.

Consider the lower bound \eqref{lowerB}. The condition is met if the following inequality is true:
\begin{equation}
\label{lowerB2}
\min_{1\leq i \leq d} |x_{ui}-x_{vi}|_T = \dist{x_u - x_v}_\text{min}
< c_L^{1/\alpha}\frac{1}{(Y_{uv})^{1/\alpha}}\cdot\left( \frac{w_uw_v}{n}\right).
\end{equation}
By the definition of MCD, it is satisfied if and only if for at least one
coordinate $1\leq i \leq d$, the absolute difference $|x_{ui}-x_{vi}|_T$ is
smaller than the right hand side. The coordinate values are i.i.d.\! from $[0,1]$; 
conditioned on $Y_{uv}$, we can consider drawing the positions as $d$ independent 
``attempts'' to fulfil  \eqref{lowerB2}. Under that view, a $d$-dimensional
MCD graph is lower-bounded by a union of $d$ one-dimensional EuGIRG instances
which all use the same $Y_{uv}$ and are thus not independent. 
However, we can make them independent by the ``splitting'' each $Y_{uv}$ into two independent random
variables in the following way.
Let $Y_{uv}^1$ and $Y_{uv}^2$  be i.i.d. from $[0,1]$, with the following CDF
for $c \in [0,1]$:
\begin{equation}
\label{CDF}
\mathbb{P} \left[Y_{uv}^1 < c\right] = \mathbb{P} \left[Y_{uv}^2 < c\right] = 1 -\sqrt{1-c}.
\end{equation}
A quick calculation shows that the following two statements hold for each $c \in [0,1]$:
\begin{align}
\label{pBound}
&c/2 \leq \mathbb{P} \left[Y_{uv}^1 < c\right] = \mathbb{P} \left[Y_{uv}^2 < c\right] \leq c &\text{and}\\
&\mathbb{P}\left[\min\{Y_{uv}^1, Y_{uv}^2\} < c \right] = \mathbb{P}\left[Y_{uv} < c\right] = c.&
\end{align}
Because of the latter, instead of checking $Y_{uv} < p_{uv}(x_u, x_v)$ for
including edges, we can equivalently use
\begin{equation} 
\label{EPs}
\tag{EIC}
\min\{Y_{uv}^1, Y_{uv}^2\} < p_{uv}(x_u, x_v)
\end{equation}
as \textbf{E}dge \textbf{I}nsertion \textbf{C}riterion.
The lower bound \eqref{lowerB} becomes
\begin{equation}
\label{lowerB3}
\min\{Y_{uv}^1, Y_{uv}^2\} < p_{uv}^L(\dist{x_u - x_v}_\text{min}).
\end{equation}
Since we have
\begin{equation}
\dist{x_u - x_v}_\text{min} =  \min\{ \min_{1\leq i < d}\{|x_{ui}-x_{vi}|_T\}, |x_{ud}-x_{vd}|_T \},
\end{equation}
Inequality \eqref{lowerB3} is satisfied if  
\begin{align}
Y_{uv}^1 &< p_{uv}^L(\min_{1\leq i < d}\{|x_{ui}-x_{vi}|_T\}) &\text{ or} \label{LB1} \tag{LB1}\\
Y_{uv}^2 &< p_{uv}^L(|x_{ud}-x_{vd}|_T).                                     \label{LB2} \tag{LB2}
\end{align}
Note that \eqref{LB1} and \eqref{LB2} are independent of each other; furthermore,
each of the two implies~\eqref{lowerB3} and thus \ref{EPs}, giving a
lower bound. Using \eqref{LB1}
as the edge insertion criterion satisfies  \eqref{EP} on
$\mathbb{T}^{d-1}$; this enables us to apply the theorems from the
previous section on the resulting graphs. For fixed $x_{u}$ and $x_{v}$, the
probability of satisfying $Y_{uv}^1 < p_{uv}^L$ is bounded by
\begin{equation}
\label{irregB}
p_{uv}^L/2 \leq \mathbb{P}\left[ Y^1_{uv} <  p_{uv}^L \right]
\leq p_{uv}^L
\end{equation}
according to  \eqref{pBound}. Thus, given that $p_{uv}^L$ satisfies \eqref{EP},
so does \eqref{LB1}.

Note that for $d = 1$, Inequality \eqref{LB1} is not defined.
This is what sets the case $d=1$ (which does have small separators) apart from $d>1$
(which does not). It does not seem to be possible to lower-bound a  one-dimensional MCD graph
with two completely independent random graphs that fulfil \eqref{EP}, which is
crucial for the proof.

\subsection{Uncovering Procedure}
\label{secProc}


Let us describe the procedure we use for uncovering the edges of an MCD-GIRG, as outlined in Section~\ref{secComp}. A condensed version can be found in Algorithm \ref{SampleAlg}.
\paragraph*{Phase 1}
We start by drawing $Y_{uv}^1$ and $x_{ui}$ for all $u, v \in V$
and all $1 \leq i \leq d-1$, i.e. leaving out all $Y_{uv}^2$ and $x_{ud}$. This is
already sufficient to determine the graph induced by  \eqref{LB1}, which
we call $G_1$.
By Theorem \ref{compSize}, there is a constant $s_\text{max}$ such
that w.h.p.\! the largest component (the giant) of $G_1$ has at least $s_\text{max} n$
vertices; we will always assume that this holds.
We denote the giant of $G_1$ by $K_\text{max}^1$.
\paragraph*{Phase 2}
According to \eqref{PL2}, there is a constant $B^\prime$ such that w.h.p.
at least $s_\text{max}n/2$ vertices of $K_\text{max}^1$ have weight less than
$B^\prime$. This can be shown using the power-law requirements: In \eqref{PL2},
set $\eta = 1$ and $B^\prime > (2c_2/s_\text{max})^{1/(\beta-2)}$. Then, at most
$s_\text{max}n/2$ vertices have weight at least $B^\prime$. Even if all of them are
in $K_\text{max}^1$, since we assume  $|K_\text{max}^1| \geq s_\text{max}n$,
at least $s_\text{max}n/2$ vertices of $K_\text{max}^1$ must
have weight smaller than $B^\prime$. (As mentioned in the previous phase, that
assumption holds with high probability.)
We now draw a set of vertices $F^\prime$ by including each
vertex (not just those in the giant) with weight less than $B^\prime$ independently
with probability $4f/s_\text{max}$, where $0 < f < \min\{\delta/12, s_\text{max}/12\}$
is a constant to be determined later.
\paragraph*{Phase 3}
We set $F := F^\prime \cap K_\text{max}^1$. The expected size of $F$ is at least
$4f/s_\text{max} \cdot s_\text{max}n/2 = 2fn$, and at most $4f/s_\text{max}
\cdot s_\text{max}n = 4 fn$. Since vertices of $K_\text{max}^1$ with weight less
than $B^\prime$ are included independently in $F^\prime$ (and thus $F$), we
can use the Chernoff bounds (Theorem \ref{wchb} with $\varepsilon = 1/2$)
to show that w.h.p.\! we have $6fn \geq |F| \geq fn$; in our further considerations,
we will assume this to be true. Note that the order of phases 1 and 2 is
interchangeable, as they are independent (the choice of $B^\prime$ depends only
on $s_\text{max}$ and not on the outcome of $G_1$).\medskip

The final three phases consist of drawing $x_{ud}$ and $Y_{uv}^2$ for all vertices
in a particular order that depends on the outcome of phases 1--3. We enumerate
the vertices from $1$ to $n$; the $k$-th vertex is denoted by $u_k$.
The enumeration has to be such that the vertices of $V \backslash K_\text{max}^1$
come first, then the vertices of $K_\text{max}^1 \backslash F$, and finally the
vertices of $F$. More precisely, we have $u_i \in V \backslash K_\text{max}^1$
for all $1 \leq i \leq |V \backslash K_\text{max}^1|$;
for all $|V \backslash K_\text{max}^1| < i \leq |V\backslash F|$ we have $u_i \in
K_\text{max}^1 \backslash F$, and finally, for all $|V\backslash F| < i \leq n$
we have $u_i \in F$.
We draw the
$d$-th coordinates step-by-step in the order given by the enumeration. The 
$d$-th coordinate of $u_k$ is written as $x_{kd}$, since the more consistent
$x_{u_kd}$ might be somewhat hard to read. Together with $x_{kd}$, we draw $Y_{kj}^2$
for all $1 \leq j < k$ (where $Y_{kj}^2$ is short for $Y_{u_ku_j}^2$). Then, for all
vertex pairs $(u_i, u_j)$ with $i, j \leq k$, we have complete information; we
add all edges fulfilling  \eqref{EPs} between such pairs. 
Therefore, the $k$-th step can be described as
\begin{itemize}
\item Draw $x_{kd}$ independently and uniformly from $X_d := [0,1]$.
\item For all $1 \leq j < k$, draw $Y_{kj}^2$ independently according to \eqref{CDF}.
\item For all $1 \leq j < k$, add the edge between $u_k$ and $u_j$ if 
\eqref{EPs} is satisfied.
\end{itemize}
The definition of the final three phases reflects the way we chose to enumerate
the vertices:
 
\paragraph*{Phase 4}
This phase consists of steps $1$ to $|V\backslash K_\text{max}^1|$, i.e. it encompasses
the following actions:
Draw $x_{ud}$ and $Y_{uv}^2$ for all $u \not = v \in V \backslash K_\text{max}^1$,
Then, for all pairs $(u,v) \in (V \backslash K_\text{max}^1 \times V \backslash K_\text{max}^1)$,
add the edge between $u$ and $v$ if  \eqref{EPs} is satisfied.
This results in a graph we call $G_2$.
\paragraph*{Phase 5}
This phase consists of steps $|V\backslash K_\text{max}^1|+1$ to $|V\backslash F|$, i.e. it encompasses
the following actions:
Draw $x_{ud}$ and $Y_{uv}^2$ for all $u \not = v \in K_\text{max}^1 \backslash F$.
Then, for all pairs $(u,v) \in (V \backslash F \times V \backslash F)$,
add the edge between $u$ and $v$ if  \eqref{EPs} is satisfied.
This results in a graph we call $G_3$.
\paragraph*{Phase 6}
This phase consists of steps $|V\backslash F|+1$ to $n$, i.e. it encompasses
the following actions:
Draw $x_{ud}$ and $Y_{uv}^2$ for all $u \not = v \in K_\text{max}^1 \backslash F$.
Then, for all pairs $(u,v) \in (V \times V)$,
add the edge between $u$ and $v$ if  \eqref{EPs} is satisfied.
This results in a graph we call $G_4$, i.e. the ``actual'' random graph from the GIRG model.

\begin{algorithm}
  \caption{Sample a MCD model graph}
  \label{SampleAlg}
\begin{algorithmic}
  \Procedure{SampleMCD}{$d, \alpha, \beta, n, \mathbf{w}, \delta$}
    \State
    \State \textit{// Phase 1}
    \State Draw $x_{ud}, Y_{uv}^1$ for all $u, v \in V$.
    \State Add edges between pairs satisfying \eqref{LB1} $\rightarrow G_1$
    \State $K_\text{max}^1 \gets $ giant of $G_1$ ($|K_\text{max}^1| \geq s_\text{max}n$ w.h.p)
    \State
    \State \textit{// Phase 2}
    \State Choose $f < \min\{\delta/12, s_\text{max}/12\}$
    \State $F^\prime \gets 0$
    \ForAll{$u \in V$ with $w_u < B^\prime$}
        \State Add $u$ to $F^\prime$ with probability $4f/s_\text{max}$
    \EndFor
    \State
    \State \textit{// Phase 3}
    \State $F \gets F^\prime \cap K_\text{max}^1$
    \State Enumerate vertices from $1$ to $n$ such that
        \State \hspace{\algorithmicindent} $\{u_1, \ldots, u_{|V \backslash K_\text{max}^1|} \} = V \backslash K_\text{max}^1$,
        \State \hspace{\algorithmicindent} $\{u_{|V \backslash K_\text{max}^1|+1}, \ldots, u_{|V \backslash F|} \} = K_\text{max}^1 \backslash F$,
        \State \hspace{\algorithmicindent} $\{u_{|V \backslash F|+1}, \ldots, u_n \} = F$
    \State (Shorthand $x_{kd}$ and $Y_{kj}^2$ for $x_{u_kd}$ and $Y_{u_ku_j}^2$, respectively)
    \State
    \State \textit{// Phase 4 (results in $G_2$)}
    \ForAll{$1 \leq k \leq |V\backslash K_\text{max}^1|$}
        \State \Call{Step}{k}
    \EndFor
    \State
    \State \textit{// Phase 5 (results in $G_3$)}
    \ForAll{$|V\backslash K_\text{max}^1| < k \leq |V\backslash F|$}
        \State \Call{Step}{k}
    \EndFor
    \State
    \State \textit{// Phase 6 (results in $G_4$)}
    \ForAll{$|V\backslash F| < k \leq n$}
        \State \Call{Step}{k}
    \EndFor

  \EndProcedure
  \State
  \Procedure{Step}{$k$}
        \State Draw $x_{kd}$ uniformly from $[0,1]$.
        \ForAll{$1 \leq j < k$}
            \State Draw $Y_{kj}^2$ independently according to the CDF given in  \eqref{CDF}.
            \State Add edge between $u_k$ and $u_j$ if  \eqref{EPs} is satisfied.
        \EndFor
    \EndProcedure
\end{algorithmic}
\end{algorithm}

Our goal is to show that the sixth and last phase destroys all ``small''
cuts in the giant of $G_3$ (Lemma \ref{lCut}) while adding few vertices to it
(Lemma \ref{l4}).
For $1 \leq i \leq 4$, we let $K_\text{max}^i$ denote the connected component
that contains $K_\text{max}^1$. 
For proving that ``small'' cuts are destroyed, we will lower-bound the probability
that an edge is present using
\begin{equation}
\tag{\ref{LB2}}
Y_{uv}^2 < p_{uv}^L(|x_{ud}-x_{vd}|_T).
\end{equation}


\section{Proof of Theorem~\ref{mainThm}}
\label{proofSec}
With all these tools laid out, we are ready to tackle the ``core'' of the proof. We
start by stating the following crucial lemma.
\begin{lemma}
\label{lCut}
There is a constant $\eta > 0$ such that with high probability, $K^3_\text{max}$
has no $(\delta, \eta)$-cut in $G_4$. (More precisely, the subgraph of $G_4$
induced by $K^3_\text{max}$ has no $(\delta, \eta)$-cut.)
\end{lemma}

In order to prove this, we need the following auxiliary Lemma \ref{l1}. 
\begin{lemma}
\label{l1}
Consider the set of $d$-th vertex coordinates $\{x_{ud} \mid u \in V\}$. It is
a set of $n$ random variables drawn uniformly and independently from $X_d = [0,1]$.
We divide $X_d$ into $M := \ceil{n/l}$ \emph{subintervals} of equal length,
where $0< l < 1$ is a constant. For all $0 \leq j < M$, we define the $j$-th subinterval
as $I_j := [j/M, (j+1)/M]$. Then, for every $0 < \delta < 1$, and every $l$, there is a constant
$r(\delta, l) > 0$ having the following property:
With probability $1-e^{-\Theta(n)}$, there is no set $S$ of $rn$ subintervals such 
that there are at least $\delta n/2$ vertices $u$ with $x_{ud} \in S$.
\end{lemma}
\begin{proof}
Fix a set $S$ of $rn$ subintervals. The sum of the lengths of all subintervals 
in $S$ is $rn \cdot \frac{1}{M} =rn / \ceil{ n/l}$. Therefore, for every
vertex $u$, the probability of $u$ getting placed in $S$ (i.e. having $x_{ud} \in S$)
is at least $rl/2$, at most $2rl$
(by virtue of $rl/2 \leq rn / \ceil{ n/l} \leq 2rl$), and for $\mu$, the expected
number of vertices placed in $S$, we have $ 2rln \geq \mu \geq rln/2$. 

Since the $x_{ud}$ are drawn independently, we can use Chernoff's inequality to bound the 
probability that at least $\delta n/2$ vertices are placed in $S$.
More precisely, we use the ``strong'' Chernoff bound (Theorem \ref{schb}). 
Note that in our case $\delta/2 \cdot n \geq \delta/2 \cdot \frac{1}{2rl} 
\mu = \frac{\delta}{4rl}\mu$, so
we set $(1+\epsilon) = \frac{\delta}{4rl}$ in Theorem \ref{schb}. We obtain
\begin{equation}
\mathbb{P} \left[ \#\{u \mid x_{ud} \in S\} \geq 
\frac{\delta}{2}n \right]
\leq \left( \frac{4erl}{\delta} \right)^{\frac{\delta}{4rl}\cdot rln/2}
=  \left[ \left( \frac{4erl}{\delta} \right)^{l \delta/16} \right]^{2n/l}
\end{equation}
If we choose $r$ small enough such that the term in square brackets is smaller 
than $1/2$, this upper bound is at most $e^{-(\log 2) 2 n/l}$.

As a final step, we take the union bound over all possible choices of $S$. 
There are at most $2^M \leq 2^{2n/l}$ ways to choose $S$, so the probability 
that for at least 
one such $S$ there are at least $\delta n/2$ vertices in $S$ is upper-bounded by
$e^{(1-\log 2 ) 2n/l} = e^{-\Theta(n)}$.
\end{proof}

Note that restricting the set of vertices considered for Lemma \ref{l1} just makes
the statement weaker; the lemma implies that for any $V^\prime \subseteq V$, w.h.p.
there is no set $S$ of $rn$ subintervals such that there are at least $\delta n/2$
vertices $u \in V^\prime$ with $x_{ud} \in S$. We will make use of this for the 
following corollary that will provide a lower bound for the edge probabilities in the proofs
of Lemmas \ref{lCut} and \ref{l4}.

\begin{corollary}
\label{cor1}
There is a constant $P >0$ such that with high probability, the following holds
for each step $k$ of Algorithm \ref{SampleAlg}, where $\delta n/2 <  k \leq n$:\\
Let $V_k := \{v_i \in V \mid 1\leq i < k\}$. For each subset  $A \subseteq V_k$
of size at least $\delta n/2$, the probability that step $k$ 
produces an edge from $u_k$ to $A$ due to \eqref{LB2} is at least $P$, i.e.
\begin{equation}
\mathbb{P}\left[\exists v \in A \text{ s.t. } u_k \sim v\right] \geq P.
\end{equation} 
\end{corollary}
\begin{proof}
Choose $l:= w_\text{min}^2$.
With that choice of $l$, any two vertices $u \neq v$ in the same subinterval have distance at most $|x_{ud} - x_{vd}|_T \leq l/n$, so in particular
$p_{uv}^L(\dist{x_u - x_v}_\text{min}) \geq c_L$. By~\eqref{pBound}, this means that
\begin{equation}
\label{closeB1}
\mathbb{P}\left[ Y_{uv}^2 < p_{uv}^L \mid u,v \text{ in the same subinterval}\right]
\geq c_L/2
\end{equation}
as a probability over $Y_{uv}^2$. 

Fix a $\delta n/2 <  k \leq n$. Apply  Lemma \ref{l1} to $V_k$; this yields
that with probability $1 - e^{\Theta(n)}$, there is no set $S$ of $rn$ subintervals
such that there are at least $\delta n/2$ vertices $v \in V_k$
with $x_{vd} \in S$. In particular, this implies that for every subset $A \subseteq
V_k$ of size at least $\delta n/2$, there are more than $rn$
subintervals $I_j$ that contain at least one vertex of $A$ (i.e. there is at least
one $v \in A$ such that $x_{vd} \in I_j$). Thus, with probability at least
$rn/M = rn/\ceil{n/l}$, the position $x_{uv}$ is in a subinterval
that contains at least one vertex of $A$. If $x_{ud}$ is in the same subinterval
as $x_{vd}$ for some $v \in A$, we have $\mathbb{P}[u\sim v] \geq c_L/2$ (see~\eqref{closeB1}), and thus
\begin{equation}
\mathbb{P}\left[\exists v \in A \text{ s.t. } u_k \sim v\right] \geq rn/\ceil{n/l} \cdot c_L/2 \geq P,
\end{equation}
for some constant $P > 0$. Note that this statements holds independent of the $Y_{uv}^1$ and $x_{ui}$ for $0 \leq i < d$. 

The probability that Lemma \ref{l1} fails for a particular $k$ is at most
$e^{-\Theta(n)}$. The claim thus follows by a union bound over all $k$.
\end{proof}


We are now equipped to complete the proof of Lemma \ref{lCut}:
\begin{proof}[Proof of Lemma \ref{lCut}]
At the end of Phase 5, consider bipartitions of $K^3_\text{max}$ into two sets $C_1$ and $C_2$,
both containing at least $\delta n$ vertices. 
Let $\mu := P \cdot fn/2$, where $P$ is
the same constant as in Corollary \ref{cor1}.
We apply Lemma \ref{LMD} to $K^3_\text{max}$ with an $\varepsilon > 0$
such that $\log (1 + \varepsilon) < Pfn/16 = \mu/(8n)$. The lemma provides an
$\eta^\prime > 0$ such that there are at most  $(1+\varepsilon)^n$
bipartitions of $K^3_\text{max}$ into two subsets $C_1$ and $C_2$ with
at most $\eta^\prime n$ cross-edges at the end of Phase 5.
In particular, this is an upper bound for the number of bipartitions into two
sets which both have size at least $\delta n$, with at most $\eta^\prime n$
cross-edges.

Furthermore, a partition of $K^3_\text{max}$ also induces a bipartition of $F$.
Since $F$ has at least $fn$ vertices w.h.p., we have for at least one $i \in
\{ 1, 2 \}$ that $|F \cap C_i| \geq fn/2$; we may assume that this holds for $i = 1$.
As we require $f < \delta/12$ (and $F$ has at most $6fn$ vertices w.h.p.),
we have $|C_2 \backslash F| \geq \delta n / 2$. We can therefore
apply Corollary \ref{cor1} to obtain that in Phase 6, each vertex in $F \cap
C_1$ receives at least one edge to $C_2 \backslash F$ with probability at least
$P = \Theta(1)$. Importantly, for any pair of
vertices $v, w \in F \cap C_1$ the events ``$v$ has at least one edge to
$C_2 \backslash F$'' and ``$w$ has at least one edge to $C_2 \backslash F$'' are
independent (because we implicitly condition on the positions of the vertices 
in $C_2 \backslash F$), so we may apply a Chernoff bound. In particular, note 
that $\mu = P \cdot fn/2$ is a lower bound on the expected number of 
vertices in $F \cap C_1$ which have at least one edge to $C_2\backslash F$ 
(and as such also a lower bound on the number of edges from $C_1$ to 
$C_2    $). By the Chernoff bound (Theorem \ref{wchb} with $\varepsilon 
= 1/2$), with probability at least $1 - e^{-\frac{1}{8}\mu}$ 
there are at least $\mu/2$ such edges.

Finally, let $\eta := \min \left\{ \eta^\prime, \mu/(2n) \right \}$. According to
the first part, there are at most $(1+\varepsilon)^n$ bipartitions of $K^3_\text{max}$
into two sets such that both have size at least $\delta n$, and have at most
$\eta n \leq \eta^\prime n$ cross-edges. According to the second part, any such
bipartition has at least $\mu/2 \geq \eta n$ cross-edges after Phase 6 with
probability not less than $1 - e^{-\frac{1}{8}\mu}$. Therefore, by a
union bound over all such bipartitions, the probability that there
is a $(\delta, \eta)$-cut (i.e. a bipartition with less than $\eta n$ cross-edges)
in the subgraph of $G_4$ induced by $K^3_\text{max}$ is at most 
$(1+\varepsilon)^n e^{-\frac{1}{8}\mu} = e^{-\Theta(n)}$. 
\end{proof}

What remains is to show that going through Phase 6 does not add too many vertices to the giant.
\begin{lemma}
\label{l4}
There is a constant $f(\delta) > 0$ such that
with high probability
\begin{equation}
\left | K^4_\text{max} \right | \leq \left | K^3_\text{max} \right | + 3 \delta n. 
\end{equation}
Note that trivially $K^3_\text{max} \subseteq K^4_\text{max} $.
\end{lemma}

We will prove Lemma~\ref{l4} by showing that for an appropriate $f$, after Phase 5 there are few vertices 
outside of $K^3_\text{max}$ that are in large components, or in components which
have at least one ``heavy'' vertex. So these components cannot increase $K^4_{\text{max}}$ by much. Afterwards, we will show 
that $F$ has few edges to vertices which fall in neither of these categories.

Let us begin with showing that few vertices are in large components which are not the giant:
\begin{lemma}
\label{l2}
There is a constant $s_t$ such that with high probability, the graph $G_3$ has
less than $\delta n$ vertices which are in a component that a) is not the giant,
and b) has size at least $s_t$.
\end{lemma}
\begin{proof}
Choose $s_t$ to be larger than $4/(s_\text{max}P)$, where $P$ is the same as in
Corollary \ref{cor1}.
Consider Phase 5 of Algorithm \ref{SampleAlg}, which is used to generate $G_3$. 
It consists of steps from $|V\backslash K_\text{max}^1|+1$ to $|V\backslash F|$. Recall that
\begin{equation}
\left\{ u_k \in V \mid |V\backslash K_\text{max}^1| < k \leq |V\backslash F| \right\} = K_\text{max}^1 \backslash F.
\end{equation}
Let $A_k$ be the set of vertices in $G_2$ which have the following property:
At the end of step $k-1$,
they are in a component of size at least $s_t$ but not in the giant.
For each step $k$ of Phase 5,  we define the indicator random variable $Z_k$ which is
1 if and only if $A_{k+1} \not = A_k$, i.e. if and only if step $k$ adds an edge
from $u_k$ to at least one vertex in $A_k$ (thus attaching it to the giant).
We apply  Corollary \ref{cor1} yielding that with high probability, at each step $k$ with $|A_k| \geq \delta n$, 
vertex $u_k$ receives an edge to $A_k$ with  probability at least $P$, i.e.
\begin{equation}
\mathbb{P}\left[Z_k  \mid  |A_k| \geq \delta n\right] \geq P\; ;
\end{equation}
we call this a ``hit''. Additionally, we define indicator random variables $Z^\prime_k$
as follows:
\begin{equation}
Z^\prime_k :=
\begin{cases}
Z_k \text{ , if $|A_k| \geq \delta n$}\\
B_{P} \text{ , otherwise.}
\end{cases}
\end{equation} 
Here, $B_{P}$ is a indicator random variable with success probability $P$.
Note that while the $Z^\prime_k$ are not completely independent, we have
$\mathbb{P}[Z^\prime_k =1 \mid A_k]  \geq P$ for all $k$ and all $A_k \subseteq V_k$, i.e.
not depending on the outcome of the previous steps.
Therefore, we may lower-bound
the sum over the $Z^\prime_k$ by a sum over $|K^1_\text{max} \backslash F|$ independent
instances of $B_{P}$. We get
\begin{equation}
\mathbb{E}\left[\sum_{k = |V\backslash K^1_\text{max}|+1}^{|V \backslash F|} Z^\prime_k\right] \geq 
\mathbb{E}\left[\sum_ {k = |V\backslash K^1_\text{max}|+1}^{|V \backslash F|} B_{P}\right] > P 
\cdot s_\text{max} n /2 =: \mu.
\end{equation}
For the last step, we use linearity of expectation and the fact that we require
$f < s_\text{max}/12$ (and $F$ has at most $6fn$ vertices w.h.p),
which implies $|K^1_\text{max} \backslash F| > s_\text{max} n /2$.
Applying the Chernoff bound (Theorem \ref{wchb}) to $\sum_k B_{P}$ yields 
that w.h.p. $\sum_k  B_{P} \geq \mu/2$ and therefore also that w.h.p. $\sum_k 
Z^\prime_k \geq \mu/2$. Now, under the assumption of $\sum_k Z^\prime_k \geq \mu/2$,
there are two possibilities: If $A_{|V \backslash F|+1}$ is smaller than 
$\delta n$, we are done (as this is exactly the set of vertices in components
of size at least $s_t$ but not the giant in $G_3$). Otherwise, each case of 
$Z^\prime_k = 1$ actually resulted from a ``hit'', so at least $s_t \cdot \mu/2$ 
vertices were eliminated from $A$. However, by our choice of $s_t > 4/(s_\text{max}P) = 2n/\mu$,
this leads to a contradiction. Therefore, under the assumption of $\sum_k Z^\prime_k \geq \mu/2$ (and thus
with high probability),
the set $A_{|V \backslash F|+1}$ has less than  $\delta n$ vertices
as desired.
\end{proof}

Furthermore, most vertices in small components also have small weight:
\begin{lemma}
\label{l3}
There is a constant $B > 0$ such that there are at most
$\delta n / s_t$ vertices with weight at least $B$ which are in a component of
size at most $s_t$ in $G_3$. In particular, there are at most $\delta n$ vertices in
$G_3$ which are in a component of size at most $s_t$ and share that component
with a vertex with weight at least $B$. For convenience, we choose $B \geq B^\prime$
(so that the weight of vertices in $F$ is guaranteed to be at most $B$).
\end{lemma}
\begin{proof}
We use \eqref{PL2} to bound the total number of vertices with weight at least $B$. In
particular, we set $\eta=1$ in \eqref{PL2} and $w$ to $\max\{B^\prime,
(c_2s_t/\delta)^{1/(\beta-2)}\} =:~B$. Inequality \eqref{PL2}
then yields that there are at most $\delta n / s_t$ vertices with weight at least~$B$.
Furthermore, there are clearly no more than $\delta n$ vertices which are in a component
of size at most $s_t$ and share that component with a vertex with weight at least $B$.
\end{proof}

It only remains to bound the number of edges that go from $F$ to low-weight vertices.
\begin{lemma}
\label{lEdgeNum}
There is an $f > 0$ such that with high probability, there are at most
$\delta n /s_t$ edges from $F$ to vertices of weight at most $B$.
\end{lemma}
\begin{proof}
Recall how $F$ was defined: We include each vertex with weight less than $B^\prime$
independently with probability $4f/s_\text{max}$ in a set $F^\prime$. (The weight
bound $B^\prime \leq B$ was chosen such that there are at least $s_\text{max}n/2$
vertices with weight less than $B^\prime$ in $K_\text{max}^1$.) Then, the set $F$
is given as $F^\prime \cap K_\text{max}^1$. Note that the generation of $F^\prime$
is completely independent of sampling the actual graph, so we might as well draw
$F^\prime$ before generating even $G_1$. Using a Chernoff bound (Theorem \ref{wchb}i
with $\varepsilon = 1/2$), we can show that w.h.p we have $|F^\prime|
\leq (6fn/s_\text{max}) =: f^\prime n$. Let $S$ be the set of all vertices
which have weight at most $B$; note that $F^\prime \subseteq S$.
We will show that for small enough $f$ (and thus $f^\prime$), there are at most
$\delta n/s_t$ edges from $F^\prime$ to $S$ in $G_4$; this completes the
argument, as $F =  F^\prime \cap K_\text{max}^1 \subseteq F^\prime$. 

We will use the Azuma-Hoeffding bound with two-sided error event, Theorem~\ref{thm18}; in the following, we
describe the choice of parameters for the theorem. 
The induced subgraph $G_4[S]$ of $G_4$ is determined by the following $2|S|-1$ independent random variables. We may assume that $S = \{1,\ldots,|S|\}$, so that we have an ordering of the vertices. (It does not need to be the same ordering as in Algorithm \ref{SampleAlg}). Firstly, for every $v\in S$ we include its random position $x_v$ as a random variable. Secondly, for each $v \in S \setminus\{1\}$ we include the random variable $Y_v :=(Y_{v1},\ldots,Y_{v\, {v-1}})$, which is uniformly at random in $[0,1]^{v-1}$. The random variables $Y_{uv}$ are the same as in Section~\ref{secLB}, so we include the edge $\{u,v\}$ if and only if $Y_{uv} \leq p_{uv}$. As in Theorem~\ref{thm18}, we denote the product space corresponding to the $2|S|-1$ random variables by $\Omega$. (By abuse of notation, we also denote the sample space by $\Omega$.) For $\omega \in \Omega$, the function $g(\omega)$ is simply defined as the number of edges from $F^\prime$  to $V$ in $\omega$. 


Note that $0 \leq g(\omega) \leq |F^\prime| \cdot |S| \leq f^\prime n^2$.
This corresponds to $M$ in the theorem. We define the ``bad'' subset $\mathcal{B}$
as follows:
\begin{equation}
\mathcal{B} := \{ \omega \in \Omega \mid \exists u \in S \; : \; 
\text{deg}(u) \geq 2C\log^2 n \}
\end{equation}
where $C>0$ is the constant from Lemma~\ref{thm:degrees}. Note that by Lemma \ref{thm:degrees} (and using the fact that the weights of vertices
in $S$ are bounded by $B$), we have $\mathbb{P}[\mathcal{B}] = n^{-\omega(1)}$. Moreover, observe that for any $\omega, \omega^\prime \in \Omega \setminus {\mathcal{B}}$ 
which differ in at most two components, we have
\begin{equation}
|g(\omega) - g(\omega^\prime)| \leq 4C \log^2 n =: c
\end{equation}
since the outcome of every $x_u$ and $Y_u$ affects at most $2C\log^2 n$ edges
(because $\omega$ and $\omega^\prime$ are both in $\Omega \setminus {\mathcal{B}}$ and
all edges a single coordinate affects have a vertex in common).
This corresponds to $c$ in the theorem.

Furthermore, Lemma \ref{thm:degrees} also implies that the expected degree of each vertex in
$S$, and thus $F^\prime$, is at most $B\cdot C$. Since $g$ is upper-bounded by the sum of
degrees of vertices in $F^\prime$, we get that
\begin{equation}
\mathbb{E}[g(Z)] \leq \mathbb{E}\left[\sum_{u \in F^\prime} \text{deg}(u)\right] \leq BCf^\prime n.
\end{equation}
This allows us to set $t := 2BCf^\prime n - \mathbb{E}[g(Z)] \geq BCf^\prime n$
for Theorem  \ref{thm18}.

Putting all of this together, Theorem \ref{thm18} yields
\begin{equation}
\mathbb{P}\left[g(Z) -\mathbb{E}[g(Z)] \geq t\right] \leq 
2e^{-\frac{(BCf^\prime n)^2}{32 \cdot 2n 
(2\log^2 n)^2}}
+(2\frac{2n\cdot f^\prime n^2}{2\log^2 n} + 1)n^{-\omega(1)} = n^{-\omega(1)}. 
\end{equation}
If we choose $f$ small enough such that $2BCf^\prime < \delta/s_t$, we get the desired
statement, since for such a choice of $f^\prime$, the theorem implies 
that w.h.p. $g(Z) \leq \mathbb{E}[g(Z)] + t = 2BCf^\prime n < \delta n/s_t$.
\end{proof}

We now have all we need to complete the proof of Lemma \ref{l4}
\begin{proof}[Proof of Lemma \ref{l4}]
Each vertex in $G_3$ is an element of exactly one of the following sets:
\begin{enumerate}
\item $K^3_\text{max}$.
\item The set of vertices in a component of size at least $s_t$, but not in $K_\text{max}^3$.
\item The set of vertices in a component of size less than $s_t$
which contains at least one vertex with weight at least $B$.
\item The set of vertices in a component of size less than $s_t$ in which all 
vertices have weight less than $B$. We call this set $T$.
\end{enumerate}
Going through Phase 6 will attach some vertices from sets 2--4 to 
$K^3_\text{max}$; our goal is to show that w.h.p at most $3\delta n$ vertices are 
attached this way. Lemma \ref{l2} tells us that w.h.p. the second set contains at most 
$\delta n $ vertices. Lemma \ref{l3} gives the same result for the third set. 
Therefore these two sets can contribute at most $\delta n$ vertices each, and  
the remaining $\delta n$ must come from $T$ (set 4).  However, according to Lemma
\ref{lEdgeNum}, there are only at most $\delta n /s_t$ edges from $F$ to $T$ w.h.p. 
(since vertices in $T$ all have weight less than $B$). Each
such edge can attach at most $s_t$ vertices from $T$ to $K^3_\text{max}$, so w.h.p at
most $\delta n$ edges from $T$ are attached, completing the argument.
\end{proof}

\begin{theorem}
\label{mainThmH}
With high probability, $K^4_\text{max}$ has no $(4\delta, \eta)$-cut, where
$\eta$ is the same as in Lemma \ref{l4}.
\end{theorem}
\begin{proof}
Assume there is a $(4 \delta, \eta)$-cut of $K^4_\text{max}$; such a cut induces
a bipartition of vertices into two set $C_1$ and $C_2$ which are both of size at least
$4\delta n$. Consider the restriction of $C_1$ and $C_2$ to $K^3_\text{max}$. Lemma \ref{l4}
implies that for $i \in \{1,2\}$ we have $|  K^3_\text{max} \cap C_i| \geq 
|C_i| - 3\delta n \geq \delta n$.
Therefore, any $(4 \delta, \eta)$-cut of $K^4_\text{max}$ induces a $(\delta, \eta)$-cut of
$K^3_\text{max}$, but according to Lemma \ref{lCut}, such a cut does not exist w.h.p.
\end{proof}

Since Theorem~\ref{mainThmH} is just a slight reformulation of our main Theorem~\ref{mainThm}, this concludes the proof.

\section{Conclusion}
\label{secConc}
We studied the separator properties of the MCD model for dimensions greater than
one, and proved that this class of random graphs does not have small separators.
This is a substantial difference to the one-dimensional case and to EuGIRGs. Our result shows that the underlying geometry decides on whether there are small separators. Future
research could include investigating distance functions that interpolate between MCD and
euclidean distance. Under the MCD model, two vertices are likely to connect
if their positions are close along the first axis, OR the second axis, \ldots , OR the $d$-th axis,
whereas in Euclidean GIRGs, they have to be close along the first axis, AND the second axis,
\ldots , AND the $d$-th axis. Being able to easily mix and match these two would allow for formulating
intuitive models, with connection criteria resembling boolean formulae (such as
``Two people are likely to be friends if they both like football AND live close to
each other, OR if both like the same particular online game'').
Hopefully, the tools laid out in this paper for analyzing the MCD model will
prove useful for these questions.

\appendix

\section{Appendix: Clustering Coefficient}\label{sec:clustering}

In this appendix we show that w.h.p. the clustering coefficient of an MCD-GIRG is $\Omega(1)$. In fact, our result holds for any GIRG model for which the distance function satisfies a \emph{stochastic triangle inequality}, i.e. two uniformly random points in the $\eps$-neighbourhood of some position $x \in \mathbb{T}^d$ have probability $\Omega(1)$ to have distance at most $O(\eps)$. The proof follows closely the proof for EuGIRGs from \cite{BKLeuclidean}, except that we replace the deterministic triangle inequality by its stochastic version. We use the notation $u\sim v$ to indicate that $u$ and $v$ are adjacent.

For the definition of clustering coefficient, there are several slightly different definitions. We follow the convention in~\cite{WattsStrogatz}. 
\begin{definition}\label{def:cc}
For a graph $G=(V,E)$ the clustering coefficient of a vertex $v\in V$ is defined as
\[
\cc(v) := \begin{cases}\frac{1}{\binom{\deg(v)}{2}}\cdot |\text{triangles in $G$ containing $v$}|, & \text{ if } \deg(v) \geq 2, \\ 0, & \text{ otherwise,}\end{cases}
\]
and the (mean) clustering coefficient of $G$ is $\cc(G) := \frac{1}{|V|} \sum_{v\in V}\cc(v)$.
\end{definition}

In case of GIRGs, we can use that the edges are sampled independently after the positions are fixed, to get the following expression for $\Ex[\cc(G)]$. 

\begin{lemma}\label{lem:localclustering}
Let $G$ be a GIRG, and let $v \in V$ be a vertex. 
Then 
\[
\Ex[\cc(v)] = \Pr[\deg(v)\geq 2]\cdot \frac{1}{\binom{n-1}{2}}\sum_{u,u' \in V\setminus \{v\}, u\neq u'}\Pr[u\sim u' \mid u \sim v, u' \sim v].
\]
\end{lemma}
\begin{proof}
Fix the vertex $v$, and let $u, \tilde u \in V\setminus \{v\}$ uniformly at random subject to $u \neq \tilde u$. Moreover, let $\mathcal E_{N}$ be the event that $u$ and $\tilde u$ are both adjacent to $v$, and let $\mathcal E_\Delta$ be the event that $v, u, \tilde u$ form a triangle. Then we may write $\Ex[\cc(v)] = \binom{n-1}{2}/\binom{\deg(v)}{2} \cdot \Pr[\mathcal E_\Delta]$, and for any $s \geq 2$ we have $\Pr[\mathcal E_N \mid \deg(v) =s] = \binom{s}{2} / \binom{n-1}{2}$. Hence,
\begin{align*}
\Ex[\cc(v)]  & = \sum_{s=2}^{n-1}\Pr_G[\deg(v)=s]\Pr_{G,u,\tilde u}[\mathcal E_\Delta \mid \mathcal E_N, \deg(v)=s].
\end{align*}
Now recall that in a GIRG all edges are independent after fixing the geometric positions. Hence, we have $\Pr[\mathcal E_\Delta \mid \mathcal E_N, x_v = x, \deg(v)=s] = \Pr[\mathcal E_\Delta \mid \mathcal E_N, x_v = x]$ for every position $x \in \mathbb{T}^d$. Consequently, we also get the same for the marginal probabilities, $\Pr[\mathcal E_\Delta \mid \mathcal E_N, \deg(v)=s] = \Pr[\mathcal E_\Delta \mid \mathcal E_N]$. The claim follows.
%
\end{proof}

We will show that the following property is enough to ensure strong clustering. 

\begin{definition}\label{def:stochtriangle}
A distance function $\dist. : \mathbb{T}^d \to \mathbb{R}_{\geq 0}$ in the sense of Section~\ref{secModel} satisfies a \emph{stochastic triangle inequality} if there is a constant $C>0$ such that the following two conditions hold. 
\begin{enumerate}
\item For every $\eps >0$ let $x_1 = x_1(\eps),x_2 = x_2(\eps)$ be uniformly at random in the $\eps$-ball $\{x \in \mathbb{T}^d \mid \dist{x} \leq \eps\}$. Then
\begin{align}\label{eq:stochtriangle}
\liminf_{\eps \to 0} \Pr[\dist{x_1-x_2} \leq C \eps] > 0.
\end{align}
\item Moreover,
\begin{align}\label{eq:stochtriangle2}
\liminf_{\eps \to 0} \frac{\Vol(\{x \in \mathbb{T}^d \mid \dist{x} \leq \eps\})}{\Vol(\{x \in \mathbb{T}^d \mid \dist{x} \leq C \eps\})}> 0.
\end{align}
\end{enumerate}
\end{definition}
We note that every metric satisfies a stochastic triangle inequality for $C=2$, since the probability in~\eqref{eq:stochtriangle} is one by the triangle inequality, and~\eqref{eq:stochtriangle2} is satisfied with limit $C^{-d}$. The minimum component distance satisfies a stochastic triangle inequality for $C=2$, since a random point $x_1$ in the $\eps$-ball must have at least one ``small'' component of absolute value at most $\eps$, and the probability that both $x_1$ and $x_2$ have the \emph{same} small component is at least $1/d$. Thus the $\liminf$ in \eqref{eq:stochtriangle} is at least $1/d$. Moreover,~\eqref{eq:stochtriangle2} is satisfied with limit $1/C = 1/2$.

\begin{theorem}\label{thm:clustering}
Consider the GIRG model with a distance function that satisfies the stochastic triangle inequality, and let $G$ be a random instance. Then $\cc(G) = \Omega(1)$ with high probability.
\end{theorem}

\begin{proof}
Let $V' := \{v\in V \mid w_v \le n^{1/8}\}$ and $G'=G[V']$. Considering $G'$ will be helpful later on to prove concentration, but we first show that $\Ex[\cc(G')] = \Omega(1)$.  
Since in the definition of distance function we required the volume $V(r) = \Vol(\{x \in \mathbb T^d \mid x \leq r\})$ to be surjective on $[0,1]$, there exists $r>0$ such that $V(r) = 1/n$. For any vertex $v$, we define $U(v) := \{x \in \mathbb T^d \mid \dist{x-x_v} \leq r\}$ to be the $r$-neighborhood of $x_v$. Note that by~\eqref{EPL} and~\eqref{PL1}, any two different vertices $u, v$  with $u \in U(v)$ have edge probability $p_{uv} = \Omega(1)$. 

Now we restrict the vertex set even further. By~\eqref{PL2} there exists $w_0 = O(1)$ such that there are $\Omega(n)$ vertices with weight at most $w_0$. Since $\cc(G') = \frac{1}{|V'|}\sum_{v\in V'} \cc_{G'}(v)=\Omega(\frac{1}{n}\sum_{v\in V'} \cc_{G'}(v))$, it suffices to show that a vertex $v$ of weight at most $w_0$ has expected clustering coefficient $\Ex[\cc_{G'}(v)] = \Omega(1)$. So we fix a vertex $v \in V_{\leq w_0} :=\{v\in V \mid w_v \leq w_0\}$ and its position $x_v$. 

Then the expected number of vertices in $V_{\leq w_0}$ with position in $U(v)$ is $\Theta(1)$. Consider the event~$\mathcal{E} = \mathcal{E}(v)$ that the following three properties hold.
\begin{enumerate}[(i)]
\item $v$ has exactly two neighbors in $V_{\leq w_0}$ with positions in $U(v)$.
\item $v$ does not have neighbors in $V_{\leq w_0}$ with positions in $\mathbb T^d\setminus U(v)$.
\item $v$ does not have neighbors in $V\setminus V_{\leq w_0}$.
\end{enumerate}
We claim that $\Pr[\mathcal{E}] = \Theta(1)$. For (i), note that the expected number of vertices in $V_{\leq w_0}$ with position in $U(x)$ is $\Theta(1)$. Since the position of every vertex is independent, by LeCam's theorem (Theorem~\ref{thm:lecam}) the probability that there are exactly two vertices in $V_{\leq w_0}$ with position in $U(x)$ is $\Theta(1)$. Moreover, for each such vertex the probability to connect to $v$ is $\Theta(1)$, so (i) holds with probability $\Theta(1)$. Conditioned on these events (i.e., all other vertices are in $\mathbb T^d\setminus U(v)$), we bound (ii). Recall that for any vertex $u \neq v$ we have $\Pr[u \in \mathbb T^d\setminus U(v)] = 1-1/n$. Therefore, for any vertex $u \neq v$, we can bound 
\[
\Pr[v \sim u \mid x_v \in \mathbb T^d\setminus U(v)] \leq \frac{\Pr[v \sim u]}{\Pr[x_v \in \mathbb T^d\setminus U(v)]} \leq \Pr[v \sim u] + \frac{1}{n-1} = O(w_u/n),
\]
where the last step follows since $\Pr[v \sim u] = \Theta(\min\{1,w_uw_v/n\})$~\cite[Lemma 4.3]{BKLgeneral} and $w_v = \Theta(1)$. Since $W = \sum_{u \in V}w_u = O(n)$ and $\max_{v\in V}{w_v} = o(n)$ by~\eqref{PL2}, we may apply LeCam's theorem, and (ii) holds with probability~$\Theta(1)$. Finally, for every fixed position $x$, (iii) holds independently of (i) or (ii) with probability $\Theta(1)$, again by LeCam's theorem. This proves the claim that $\Pr[\mathcal{E}] = \Theta(1)$. 

Conditioned on $\mathcal{E}$, let $u$ be a neighbor of $v$. Then $x_{u}\in U(v)$, and $w_u \leq w_0$. Moreover, since $p(w_v,w_u,x_v,x_u) = \Theta(1)$ for all weights $w_u \leq w_0$ and all positions $x_u \in U(v)$, we have for all measurable $A \subseteq U(v)$,
\begin{align*}
\Pr_{x_u}[x_u \in A \mid \mathcal{E}, v_1\sim v] & = \frac{\int_{A} \Pr[x_1=u, v\sim v_1]du}{\int_{U(v)} \Pr[x_1=u,v\sim v_1 ]du} = \Theta\left(\frac{\Vol(A)}{\Vol(U(v))}\right).
\end{align*}
In other words, conditioned on $\mathcal E$, the position $x_u$ is uniformly distributed in $U(v)$ up to constant factors. The same holds for the second neighbour $u'$ of $v$. 
Therefore, by~\eqref{eq:stochtriangle}, with probability $\Omega(1)$ we have $\dist{u-u'} \leq Cr$. By~\eqref{eq:stochtriangle2} we also have $V(Cr) = O(V(r)) = O(1/n)$, and thus $p_{uu'} = \Omega(1)$. So we have shown that $\Ex[\cc_{G'}(v)\mid \mathcal{E}(v)] = \Omega(1)$ for all $v\in V_{\leq w_0}$. Since $\Pr[\mathcal{E}(v)] = \Theta(1)$, this proves $\Ex[\cc_{G'}(v)] = \Omega(1)$ for all $v\in V_{\leq w_0}$, which implies $\Ex[\cc(G')]=\Omega(1)$. \smallskip

It remains to show concentration. We want to apply the Azuma-Hoeffding bound with two-sided error event (Theorem~\ref{thm18}) to $f(G') := \cc(G')$, where we define the error event $\mathcal{B}$ to be the event that there is a vertex in $V'$ of degree larger than $Cn^{1/8}$ for a suitable constant $C>0$. By Lemma~\ref{thm:degrees} we have $\Pr[\mathcal B] = n^{-\omega(1)}$ if $C$ is large enough. 
Similar as in the proof of Lemma~\ref{lEdgeNum}, we apply Theorem~\ref{thm18} to the following $2|V'|-1$ independent random variables, where we may assume that $V' = \{1,\ldots,|V'|\}$. Firstly, for every $v\in V'$ we include its random position $x_v$ as a random variable. Secondly, for each $v \in V' \setminus\{1\}$ we include a random variable $Y_v :=(Y_{v1},\ldots,Y_{v\, {v-1}})$, which is uniformly at random in $[0,1]^{v-1}$. The random variables $Y_{uv}$ play the same role as in Section~\ref{secLB}, so we include the edge $\{uv\}$ if and only if $Y_{uv} \leq p_{uv}$. As in Theorem~\ref{thm18}, we denote the product space corresponding to the $2|V'|-1$ random variables by $\Omega$. 

Note that if $\omega, \omega'$ are two events in $\Omega \setminus \mathcal B$ then every random variables $x_v$ or $Y_v$ that differs in $\omega$ and $\omega'$ can affect at most the cc's of those vertices which are adjacent to $v$ in either $\omega$ or $\omega'$. Hence, the effect of each coordinate on $\cc(G')$ is at most $\tfrac{1}{|V'|}2Cn^{1/8} \leq 4Cn^{-7/8}$, and thus $|f(\omega) - f(\omega')| \leq 8Cn^{-7/8}$ if $\omega$ and $\omega'$ differ in at most two components. Note that for this type of argument it is crucial that both of $\omega$ and $\omega'$ are ``good'' events. Now we may apply Theorem~\ref{thm18} with $c= 8Cn^{-7/8}$, $M=1$, $m=2|V'|-1 \leq 2n$, and $t:= n^{-1/4}$, and obtain
\[
\Pr\left[|\cc(G')-\Ex[\cc(G')]| \ge n^{-1/4}\right] \le 2e^{-\Omega(n^{1/4})} + 8Cn^{1/8}\Pr[\mathcal{B}]  = o(1).
\]
Since $\Ex[\cc(G')] = \Theta(1)$, this shows that w.h.p.\! $\cc(G')=\Omega(1)$. To conclude the same for $G$, we observe that $|V|\cdot\cc(G) \ge |V'|\cdot\cc(G') -\sum_{v \in V \setminus V'} \deg(v)$. However, an easy calculation shows that w.h.p.
\[
\sum_{v \in V \setminus V'} \deg(v)\stackrel{\text{\cite[Lemma 4.5]{BKLgeneral}}}{=}(1+o(1))\sum_{v \in V \setminus V'} w_v\stackrel{\text{\cite[Lemma 4.2]{BKLgeneral}}}{=}o(n).
\]
Since w.h.p.\! $|V'| = \Theta(n)$, this concludes the argument and proves $\cc(G)=\Omega(1)$ with high probability.
\end{proof}




\end{document}